\documentclass[11pt]{article}
\usepackage{amssymb}
\usepackage{graphics}
\usepackage{graphicx}
\usepackage[all]{xy}

\usepackage{ifthen}
\usepackage{color}

\newfont{\bbb}{msbm10 scaled\magstephalf}
\newfont{\sbbb}{msbm7 scaled\magstephalf}
\def\F{\mbox{\bbb{F}}}
\def\C{\mbox{\bbb{C}}}

\def\R{\mbox{\bbb{R}}}
\def\Z{\mbox{\bbb{Z}}}

\def\rd{\R^d}
\def\rn{\R^n}
\def\zd{\Z^d}

\def\rndu{(\rn)^*}
\def\rddu{(\rd)^*}
\def\et1{e^{2\pi i\theta_1}}

\def\vz{\underline{z}}

\def\vw{\underline{w}}
\def\zjs{|z_j|^2}

\def\D{\Delta}
\def\xd{X_1,\ldots,X_d}
\def\ld{\lambda_1,\ldots,\lambda_d}
\def\k{\mathfrak{k}}

\newtheorem{thm}{Theorem}[section]

\newtheorem{cor}[thm]{Corollary}
\newtheorem{prop}[thm]{Proposition}
\newtheorem{remark}[thm]{Remark}
\newtheorem{lemma}[thm]{Lemma}
\newtheorem{example}[thm]{Example}
\newtheorem{ass}[thm]{Assumption}

\def\squareforqed{\hbox{\rlap{$\sqcap$}$\sqcup$}}
\def\qed{\ifmmode\else\unskip\quad\fi\squareforqed}

\newcounter{sect}

\definecolor{red}{rgb}{.6,0,0}
\definecolor{green}{rgb}{0,.6,0}
\definecolor{darkgreen}{rgb}{0,0.3,0}
\definecolor{purple}{rgb}{0.5,0,0.5}
\definecolor{darkblue}{rgb}{0,0,0.7}
\definecolor{greenblue}{rgb}{0,0.4,0.5}
\newboolean{draft}
\newcommand{\cmt}[1]
{\ifthenelse {\boolean{draft}}
{{\sc \tiny \color{red} #1}}
{}}
\newcommand{\newbb}[1]
{\ifthenelse {\boolean{draft}}
{{\color{darkblue} #1}}
{#1}}
\newcommand{\newbbb}[1]
{\ifthenelse {\boolean{draft}}
{{\color{greenblue} #1}}
{#1}}
\newcommand{\inred}[1]
{\ifthenelse{\boolean{draft}}{{\color{red} #1}}{#1}}
\newcommand{\new}[1]
{\ifthenelse {\boolean{draft}}
{{\color{green} #1}}
{#1}}
\newcommand{\neww}[1]
{\ifthenelse {\boolean{draft}}
{{\color{darkgreen} #1}}
{#1}}
\newcommand{\newb}[1]
{\ifthenelse {\boolean{draft}}
{{\color{blue} #1}}
{#1}}
\newcommand{\del}[1]
{\ifthenelse {\boolean{draft}}
{{\color{magenta} #1}}
{}}
\setboolean{draft}{true}

\title{\sc Nonrational Symplectic Toric Reduction}
\author{\sc Fiammetta Battaglia and Elisa Prato}
\date{}
\begin{document}
\maketitle
\begin{abstract}In this article, we introduce symplectic reduction in the framework of nonrational toric geometry. 
When we specialize to the rational case, we get symplectic reduction for the action of a general, not necessarily closed, Lie subgroup of the torus.
\end{abstract}
\section*{Introduction}

If we have a symplectic manifold that is invariant under the Hamiltonian action of a Lie group, {\em symplectic} or {\em Marsden--Weinstein reduction}
\cite{marwei} allows to construct a lower dimensional symplectic manifold by reducing its symmetries. This fundamental operation has inspired a wide 
number of applications throughout geometry and physics.

In this article, we extend symplectic reduction to the context of nonrational toric geometry. We recall that the Delzant theorem \cite{d} establishes a
correspondence between smooth polytopes $\D\subset (\rn)^*$ and compact {\em symplectic toric manifolds}, meaning compact symplectic $2n$--manifolds
with the effective Hamiltonian action of the torus $T^n=\rn/\Z^n$. One of the remarkable features of this theorem is that it provides an explicit construction of
the symplectic manifold from the polytope, following the same principle that allows to construct a complex toric variety from a fan. When generalizing this
construction to simple convex polytopes that are not rational, the resulting spaces turn out to be {\em quasifolds} \cite{pcras,p}.  Quasifolds generalize
manifolds and orbifolds, and they are typically not Hausdorff. Locally, they are quotients of smooth manifolds by the action of countable groups. Similarly to
what happens in the smooth case, the quasifolds $M$ that one gets from the generalized Delzant construction are compact, symplectic, $2n$--dimensional, 
and are endowed with an effective Hamiltonian action of a {\em quasitorus} of dimension $n$. A quasitorus of dimension $n$
is the abelian group and quasifold given by the quotient $D^n=\rn/Q$, where $Q$ is a {\em quasilattice}, namely the $\Z$--span of a set of real 
spanning vectors in $\rn$. We refer to these spaces $M$ as {\em
symplectic toric quasifolds}.

The idea here is to reduce symplectic toric quasifolds of dimension $2n$ with respect to the action of {\em any} subgroup 
$K=\k/(\k\cap Q)\subset D^n$, $\k$ being a subspace of $\rn$. We prove that the resulting space is itself a symplectic toric quasifold, 
of dimension $2(n-\dim(\k))$, with the Hamiltonian quasitorus action of $D^n/K$ (see Theorem~\ref{reduction}). As a consequence, 
we are able to reduce {\em any} symplectic toric {\em manifold} with respect to the action of a {\em general} Lie subgroup $K\subset T^n$ 
(see Corollary~\ref{liscio}); if $K$ is not closed, the resulting space is a symplectic toric quasifold.
Notice thus that the class of symplectic toric quasifolds is closed under symplectic reduction. 
On the other hand, we know that the class of symplectic toric manifolds is not even closed under reduction with respect to the action of a subtorus, 
since the resulting space may be an orbifold.  Moreover, neither the class of symplectic toric manifolds nor the class of symplectic toric orbifolds 
is closed under reduction with respect to the action of a general Lie subgroup.

We state and prove our results for general pointed polyhedra instead of convex polytopes. We recall that a pointed polyhedron is a finite intersection
of closed half--spaces that has at least a vertex: it is a convex polytope if, and only if, it does not contain a ray. In this setting, the resulting symplectic toric 
quasifolds (and manifolds) may be noncompact.

The article is structured as follows: in the first section, we recall the generalized Delzant construction; in the second section, we prove the symplectic
reduction theorem, and in the third section, we discuss some applications. 

\section{The generalized Delzant procedure}
In this section, we briefly recall the extension of the Delzant procedure to the nonrational case \cite[Theorem~1.1]{p}, stated in the case of pointed
polyhedra, following \cite[Theorem~1.1]{cut}. 

We begin by recalling a few relevant facts on pointed polyhedra; for a more detailed account, we refer the reader to Ziegler \cite{ziegler}.
A subset $\D\subset\rndu$ is said to be a {\em polyhedron} if it is given by a finite intersection of closed half--spaces. 
If $\D\subset\rndu$ is an $n$--dimensional polyhedron with $d$ facets, then one can choose $X_1,\ldots,X_d \in\rn$ and 
$\lambda_1,\ldots,\lambda_d\in\R$ such that
\begin{equation}\label{decomp}\D=\bigcap_{j=1}^d\{\;\mu\in\rndu\;|\;\langle\mu,X_j\rangle\geq\lambda_j\;\}.\end{equation}
Each of the vectors $X_1,\ldots,X_d$ is orthogonal to one of the different $d$ facets of $\D$ and points towards its interior. 
We will conveniently refer to these vectors as {\em normal vectors} for $\D$.
A polyhedron can have at most a finite number of vertices; whenever one such vertex exists, we will say that the polyhedron is {\em pointed}.
A polyhedron is pointed if, and only if, it does not contain a line.
Moreover, a pointed polyhedron is a convex polytope if, and only if, it does not contain a ray. A dimension $n$ pointed polyhedron is 
{\em simple} if each of its vertices is contained in exactly $n$ facets. Finally, a simple pointed polyhedron is {\em smooth} if one can 
choose normal vectors so that, for each vertex, the vectors that are orthogonal to the corresponding $n$ facets form a basis of $\Z^n$.

In the nonrational case, lattices are replaced by {\em quasilattices} and tori by {\em quasitori}. A quasilattice $Q$ in $\rn$ is the $\Z$--span of a set of 
real spanning vectors, $Y_1,\ldots,Y_d\in\rn$; $Q$ is a lattice if, and only if, $d=n$. We call quasitorus the quotient $D^n=\rn/Q$. It is a dimension 
$n$ quasifold; it is a regular torus if, and only if, $Q$ is a regular lattice.
We will say that a polyhedron $\D\subset\rndu$ is {\em quasirational} with respect to a given quasilattice $Q$,
if the normal vectors for $\Delta$ can be chosen in $Q$. If $\D$ is quasirational with respect to a lattice, then it is rational in the usual sense. 
It is important to notice that any polyhedron is quasirational with respect to the quasilattice that is generated by any choice of normal vectors.

We are now ready to recall the generalized Delzant procedure. We will outline its proof; 
for additional details, we refer the reader to \cite[Theorem~1.1]{cut}. For the basic definitions and properties of quasifolds we refer the 
reader to \cite{p,kite}.
\begin{thm}\label{thmp1}
Let $Q$ be a quasilattice in $\rn$ and let $\D\subset\rndu$ be an $n$--dimensional simple pointed polyhedron
that is quasirational with respect to $Q$. Assume that $d$ is the number of facets of $\D$ and
consider normal vectors $X_1,\ldots,X_d$ for $\Delta$ that lie in $Q$. For each $(\D,\{\xd\},Q)$,
there exists a $2n$--dimensional symplectic quasifold $(M,\omega)$
endowed with the effective Hamiltonian action of the quasitorus $D^n=\rn/Q$ such that, 
if $\Phi\,\colon M\rightarrow\rndu$ is the corresponding moment mapping, 
then $\Phi(M)=\Delta$. If $\Delta$ is a convex polytope, then $M$ is compact.
\end{thm}
We say that the quasifold $(M,\omega)$ above is the {\em symplectic toric quasifold} corresponding to $(\D,\{\xd\},Q)$.

\noindent
\begin{sketchofproof}
Consider the standard linear Hamiltonian action of $T^d=\rd/\zd$ on $\C^d$, with its moment mapping $J(\vz)=\sum_{j=1}^d \zjs
e_j^*+\lambda$, $\lambda\in\rddu$ constant.
Consider the surjective linear mapping
\begin{eqnarray*}
\pi\,\colon &\rd \longrightarrow \rn,\\
&e_j \longmapsto X_j
\end{eqnarray*}
and let $N$ be the kernel of the corresponding quasitorus epimorphism $\Pi\,\colon\,T^d\longrightarrow D^n$.
The induced action of $N$ on $\C^d$ is also Hamiltonian, with moment mapping given by
$\Psi=i^*\circ J$, where $i$ is the Lie algebra inclusion $\mbox{Lie}(N)=\ker(\pi)\rightarrow\rd$. Choose the constant
$\lambda$ above to be equal to $\sum_{j=1}^d {\lambda_j} e_j^*$, with $\ld$ as in (\ref{decomp}). 
Notice that, since $\Delta$ is simple, the group $N$ acts on the level set $\Psi^{-1}(0)$ with $0$--dimensional isotropy groups. 
The quotient $\Psi^{-1}(0)/N$ is our symplectic quasifold $M$.
The induced action on $M$ of the quasitorus $D^n=\rn/Q\simeq T^d/N$ is Hamiltonian and its moment mapping is given by 
\begin{equation}\label{mm}
\Phi([\vz])=((\pi^*)^{-1}\circ J)(\vz),
\end{equation}
where $\vz\in\Psi^{-1}(0)$. It is straightforward to check that $\Phi(M)=\Delta$.
\qed\end{sketchofproof}

We conclude this introductory section by recalling from \cite[Theorem~3.2]{cx} the construction of an explicit atlas for $M$; 
this will be an essential ingredient in the proof of our main result. Similarly to what happens in the smooth case, 
we cover $M$ with an atlas that is indexed by the set of vertices of $\D$.
Take one such vertex $\nu$; we define a quasifold chart $(U_{\nu},\rho_{\nu},\tilde{U}_{\nu}/\Gamma_\nu)$, as follows.
Suppose, up to renumbering, that the vertex $\nu$ is the intersection of the facets that are orthogonal to the first $n$ normal 
vectors, $X_1,\ldots,X_n$. We can write the remaining normal vectors uniquely as follows
$$
X_j=\sum_{h=1}^{n}a_{jh}X_h,\quad j=n+1,\ldots,d.
$$
Consider now the open subsets
$$U_{\nu}=\left\{\,[\vz]\in M\,|\,  z_j\neq 0, j=n+1,\ldots,d \,\right\}.$$
Formula (\ref{mm}) implies that $|z_j|^2+\lambda_j=\sum_{h=1}^{n}(a_{jh}|z_h|^2-\lambda_h)$ for
every $\vz\in\Psi^{-1}(0)$. Then the set
$$\tilde{U}_{\nu}=\left\{\,(z_1,\ldots,z_n)\in \C^n \,|\, \sum_{h=1}^{n}(a_{jh}|z_h|^2-\lambda_h)-\lambda_j>0,j=n+1,\ldots,d \,\right\}$$
is non--empty.
If we take
$$
w_j=\sqrt{\sum_{h=1}^{n}a_{jh}(|z_h|^2+\lambda_h)-\lambda_j},\quad j=n+1,\ldots,d,
$$
then the mapping
$$
\begin{array}{cccc}
\tilde{\rho}_{\nu}\,\colon\,&\tilde{U}_{\nu}&\longrightarrow&U_{\nu}\\
&(z_1,\ldots,z_n)&\longmapsto&[z_1:\cdots:z_n:w_{n+1}:\cdots:w_d]
\end{array}
$$
induces a homeomorphism 
\begin{equation}\label{rhonu}
\rho_{\nu}\,\colon\,\tilde{U}_{\nu}/\Gamma_\nu\rightarrow U_{\nu},
\end{equation}
where
$\Gamma_\nu$ is the countable group $N\cap[(S^1)^n\times (1)^{d-n}]$.

\section{Symplectic reduction}
Let us consider a dimension $n$ simple pointed polyhedron $\Delta\subset\rndu$ that is quasirational with respect to a quasilattice $Q$. 
Take normal vectors, $\xd$, for $\D$ in $Q$ and apply the generalized Delzant procedure to $(\D,\{\xd\},Q)$. 
As we have seen, this yields a symplectic $2n$--quasifold $M$ with the effective Hamiltonian action of the quasitorus $D^n=\rn/Q$.

Let now $\mathfrak{k}$ be a nontrivial $k$--dimensional vector subspace of $\R^n$.  
The quotient $K=\mathfrak{k}/\mathfrak{k}\cap Q$ is a quasifold and a subgroup of the quasitorus  $D^n$.
It is a quasitorus itself when $\hbox{span}_{\R}(\mathfrak{k}\cap Q)=\mathfrak{k}$. 
We have the following exact sequences:
$$0\longrightarrow\mathfrak{k}\stackrel{j}{\longrightarrow}\R^n\stackrel{p}{\longrightarrow}\R^n/\mathfrak{k}\longrightarrow0$$
$$0\longrightarrow(\R^n/\mathfrak{k})^*\stackrel{p^*}{\longrightarrow}(\R^n)^*\stackrel{j^*}{\longrightarrow}\mathfrak{k}^*\longrightarrow0$$
Notice, in particular,  that $p^*((\R^n/\mathfrak{k})^*)=\ker j^*$.
Consider now the induced action of $K$ on $M$; this action is Hamiltonian and the corresponding moment mapping is given by
$\Phi_\mathfrak{k}=j^*\circ\Phi$. Consider now a value $\xi$ of this mapping; we can assume, up to a translation of $\D$, that $\xi=0$.
In accordance with the smooth case, we will define the orbit space
$$M_\mathfrak{k}=\Phi_\mathfrak{k}^{-1}(0)/K$$
to be the {\em symplectic reduced space} for the action of $K$ at the value $0$.
We will devote the rest of the section to showing that $M_\mathfrak{k}$ is a symplectic toric quasifold.

We begin by noticing that the mapping $\Phi$ sends all points of the level set
$\Phi_\mathfrak{k}^{-1}(0)$ onto the set $\D_{\mathfrak{k}}=\D\cap(\ker(j^*))$, which is therefore non--empty. 
Notice that, by using (\ref{decomp}), we get
$$
\begin{array}{ccl}
\D_{\mathfrak{k}}&=&\bigcap_{j=1}^{d}\{\;\mu\in\ker j^*\;|\;\langle\mu,X_j\rangle\geq\lambda_j\;\}\\
&=&\bigcap_{j=1}^{d}\{\;\nu\in(\R^n/\mathfrak{k})^*\;|\;\langle p^*(\nu),X_j\rangle\geq\lambda_j\;\}\\
&=&\bigcap_{j=1}^{d}\{\;\nu\in(\R^n/\mathfrak{k})^*\;|\;\langle \nu,p(X_j)\rangle\geq\lambda_j\;\}.
\end{array}
$$
Thus $\D_{\mathfrak{k}}$ itself is a polyhedron. Since $\D$ is pointed, so is $\D_{\mathfrak{k}}$. 
Now, if $d_\mathfrak{k}$ is the number of its facets, it is always possible to choose $d_\mathfrak{k}$ among the $d$ half--spaces 
above so that, up to renumbering,
\begin{equation}\label{deltakappa}
\D_{\mathfrak{k}}=\bigcap_{j=1}^{d_\mathfrak{k}}\{\;\nu\in(\R^n/\mathfrak{k})^*\;|\;\langle \nu,p(X_j)\rangle\geq\lambda_j\;\}.
\end{equation}
From now on we will make the following fundamental assumption.
\begin{ass}{\rm The induced action of $K$ on $\Phi_\mathfrak{k}^{-1}(0)$ has $0$--dimensional isotropy groups.}
\end{ass}
\begin{remark}{\rm Notice that this assumption is standard in classical symplectic reduction. It has a number of crucial
implications on the geometry of the pointed polyhedron $\D_{\mathfrak{k}}$, as explained in the following proposition.}
\end{remark}
\begin{prop}\label{free action} 
Consider the induced action of $K$ on $M$. The group $K$ acts on $\Phi_\mathfrak{k}^{-1}(0)$ with $0$--dimensional isotropy groups if, and only if, the pointed polyhedron $\D_{\mathfrak{k}}$ has dimension $n-k$, is simple and the $d_\mathfrak{k}$ half--spaces in (\ref{deltakappa}) are unique, in the sense that $\D_{\mathfrak{k}}$ is contained in the interior of the remaining $d-d_\k$.
\end{prop}
\begin{proof}
Assume first that $K$ acts on $\Phi_\mathfrak{k}^{-1}(0)$ with $0$--dimensional isotropy groups.
Suppose that the subspace $\ker j^*$ has empty intersection with the interior of $\Delta$. Then
$\Delta\cap\ker j^*\neq\emptyset$ implies that $\ker j^*$ intersect $\Delta$ in one of its faces. 
Therefore $\ker j^*$ contains a vertex of $\Delta$. This contradicts the hypothesis, since points
in $M$ that are sent to vertices of $\Delta$ are fixed by the $D^n$--action.
Thus $\ker j^*$ has non--empty intersection with the interior of $\Delta$, which implies that $\Delta_\k$
has dimension $n-k$. Now, if we take any vertex $\nu_{\mathfrak{k}}\in\Delta_{\mathfrak{k}}$, we can assume, again up to renumbering, that
$$\nu_{\mathfrak{k}}=\cap_{j=1}^{n-k}\{\;\nu\in(\R^n/\mathfrak{k})^*\;|\;\langle \nu,p(X_j)\rangle=\lambda_j\;\}.$$
Suppose now that there exists a vertex $\nu_{\mathfrak{k}}$, written as above, and an index $h\notin\{1,\ldots,n-k\}$, with the property that
$$\nu_{\mathfrak{k}}\in\{\;\nu\in(\R^n/\mathfrak{k})^*\;|\;\langle \nu,p(X_h)\rangle=\lambda_h\;\}.$$
If $h\in\{n-k+1,\ldots,d_\mathfrak{k}\}$, this would contradict simplicity; if $h\in\{d_\mathfrak{k}+1,\ldots,d\}$, 
this would contradict the second part of the thesis. Let us show that this is indeed not possible. 
Let $\hat{\vz}\in\Psi^{-1}(0)$ be such that $\Phi([\hat{\vz}])=p^*(\nu_{\mathfrak{k}})$. Then it is easy to verify that $\hat{z}_i=0$ if, and only if, 
$i=1,\ldots,n-k,h$. Since $\{p(X_1),\ldots,p(X_{n-k})\}$ is a basis
of $\R^n/\mathfrak{k}$, we can write $p(X_h)=\sum_{j=1}^{n-k}a_jp(X_j)$. Therefore the element
$Y=X_h-\sum_{j=1}^{n-k}a_jX_j$ belongs to $\mathfrak{k}$. For each $t\in\R$ and $[\vz]\in M$, we have
$$\begin{array}{l}\exp(tY)[z_1:\cdots:z_{n-k}:z_{n-k+1}:\cdots:z_h:\cdots:z_d]=\\

[e^{-2\pi i t a_1}z_1:\cdots:e^{-2\pi i t a_{n-k}}z_{n-k}:z_{n-k+1}:\cdots:e^{2\pi i t}z_h:\cdots:z_d].
\end{array}$$ 
Hence the isotropy of $K$ at the point $[\hat{\vz}]$ above has dimension at least $1$. By assumption, this is not possible.

Conversely, suppose that the pointed polyhedron $\D_{\mathfrak{k}}$ has dimension $n-k$, is simple and that the $d_\mathfrak{k}$ 
half--spaces in (\ref{deltakappa}) are unique. If we consider $[\hat{\vz}]\in \Phi_\mathfrak{k}^{-1}(0)$,
then $\Phi([\hat{\vz}])$ lies in an (open) face of $\Delta_\k$. Let $\nu_\k$ be a vertex in the closure of this face and write $\nu_k$ as above. Then, by hypothesis, $\hat{z}_j\neq0$ for all $j=n-k+1,\ldots,d$.
On the other hand, each of the vectors $\{p(X_{n-k+1}),\ldots,p(X_{d})\}$ can be uniquely expressed
as a linear combination of  $\{p(X_1),\ldots,p(X_{n-k})\}$:
$$p(X_j)=\sum_{h=1}^{n-k}a_{jh}p(X_h),\quad j=n-k+1,\ldots,d.$$
It is easy to check that the $d-n+k$ vectors $Y_j=X_j-\sum_{h=1}^{n-k}a_{jh}X_h$, $j=n-k+1,\ldots,d$, define a set of generators of $\mathfrak{k}$.
Thus, if we take any non--zero $Y\in\k$, we can write it as $Y=\sum_{j=n-k+1}^{d}b_{j}Y_{j}$, with at least one $b_j\neq 0$;
we assume, for simplicity, that $b_{n-k+1}\neq0$. Observe that
$Y=\pi\left(\sum_{j=n-k+1}^{d}b_{j}(e_j-\sum_{h=1}^{n-k}a_{jh}e_h)\right)$.
Now suppose $\exp_{D^n}(tY)\cdot[\hat{\vz}]=[\hat{\vz}]$ for some $t\in\R$. 
Then, there exists
$R=\sum_{j=1}^dr_je_j\in\R^d$, with $\pi(R)\in Q$, such that 
\begin{equation}\label{identity}
e^{2\pi i tb_{n-k+1}}e^{2\pi i r_{n-k+1}}\hat{z}_{n-k+1}=\hat{z}_{n-k+1}
\end{equation}
Since $\hat{z}_{n-k+1}\neq0$ we have that (\ref{identity}) is satisfied only for a countable set of $t\in\R$. Therefore
$K$ acts on the level set with $0$--dimensional isotropy groups.
\qed\end{proof}

Apply now the generalized Delzant construction to $(\D_\mathfrak{k},\{p(X_1),\ldots,p(X_{d_\mathfrak{k}})\},p(Q))$ and let $X_\mathfrak{k}$ 
denote the corresponding symplectic toric quasifold. We recall the construction of $X_\mathfrak{k}$ from the proof of Theorem~\ref{thmp1}.
Let $J_\mathfrak{k}(\vz)=\sum_{j=1}^{d_\mathfrak{k}}(|z_j|^2+\lambda_j)e_j^*$ be the moment mapping with respect to the standard action
of $T^{d_\mathfrak{k}}$ on $\C^{d_\mathfrak{k}}$. Consider the linear projection 
\begin{eqnarray*}
\pi_\mathfrak{k}\,\colon &\R^{d_\mathfrak{k}} \longrightarrow \rn/\mathfrak{k},\\
&e_j \longmapsto p(X_j)
\end{eqnarray*}
and notice that $\pi_\mathfrak{k}=p\circ\pi_{|\R^{d_\mathfrak{k}}}$. 
The kernel of the corresponding epimorphism
$\Pi_\mathfrak{k}\,\colon\,T^{d_\mathfrak{k}}\longrightarrow (\R^n/\k)/p(Q)$ is given by the ($d_\k-n+k$)--dimensional group
$$
\begin{array}{ccl}N_\mathfrak{k}&=&\exp\{X\in\R^{d_\mathfrak{k}}\;|\;\pi_\mathfrak{k}(X)\in p(Q)\}\\
&=&\exp\{X\in\R^{d_\mathfrak{k}}\;|\;\pi_{|\R^{d_\mathfrak{k}}}(X)\in Q+\mathfrak{k}\}.
\end{array}
$$
Its induced action on $\C^{d_\mathfrak{k}}$ is Hamiltonian, with moment mapping $\Psi_\mathfrak{k}=i^*_\mathfrak{k}\circ J_\mathfrak{k}$, where
$i_\mathfrak{k}$ is Lie algebra inclusion $\mbox{Lie}(N_\mathfrak{k})=\ker(\pi_\mathfrak{k})\rightarrow\R^{d_\mathfrak{k}}.$ Then we have
\begin{equation}\label{delzantspace}
\begin{array}{ccl}
X_\mathfrak{k}&=&\Psi_\mathfrak{k}^{-1}(0)/N_\mathfrak{k}
\end{array}.
\end{equation}
Notice that $(\R^n/\k)/p(Q)$ can be identified with $D^n/K$.
We are now ready to state our main result.
\begin{thm}\label{reduction}  Consider the induced action of $K$ on $M$ with moment mapping $\Phi_\mathfrak{k}=j^*\circ\Phi$. 
Assume that $K$ acts on $\Phi_\mathfrak{k}^{-1}(0)$ with $0$--dimensional isotropy groups.
Then the orbit space $M_\mathfrak{k}=\Phi_\mathfrak{k}^{-1}(0)/K$ is a symplectic quasifold of dimension $2(n-k)$, 
acted on by the quasitorus $D^n/K$. Moreover, $M_\mathfrak{k}$ is equivariantly symplectomorphic to the symplectic toric quasifold $X_\k$ corresponding to 
$(\D_\mathfrak{k},\{p(X_1),\ldots,p(X_{d_\mathfrak{k}})\},p(Q))$.
\end{thm}
To prove the above theorem we first need the following 
\begin{lemma}\label{rotazioni}
Consider $[\vw]\in\Phi_\mathfrak{k}^{-1}(0)$, then there exists 
$[\tilde{\vw}]\in\Phi_\mathfrak{k}^{-1}(0)$ such that
$w_j\in[0,+\infty)$ for $j=d_\mathfrak{k}+1,\ldots,d$ and
$[[\vw]]=[[\tilde{\vw}]]$ in $M_{\mathfrak{k}}$. Moreover 
$[\tilde{w}_1:\cdots:\tilde{w}_{d_{\mathfrak{k}}}]$ belongs to $X_{\mathfrak{k}}$ and, if $[\hat{w}]$ is another element in 
$\Phi_{\mathfrak{k}}^{-1}(0)$ having the same properties as $\tilde{\vw}$, then 
$[\tilde{w}_1:\cdots:\tilde{w}_{d_{\mathfrak{k}}}]=[\hat{w}_1:\cdots:\hat{w}_{d_{\mathfrak{k}}}]$ in $X_{\mathfrak{k}}$.
\end{lemma}
\begin{proof}
We begin by writing the moment mapping $\Phi_{\mathfrak{k}}$ and the action of $K$ explicitly.
In order to do so, we choose a vertex $\nu_\mathfrak{k}$ of
$\Delta_\mathfrak{k}$ and we order the normal vectors for $\Delta$ so that $\{p(X_1),\ldots,p(X_{n-k})\}$ are the normal vectors corresponding to the facets 
meeting at $\nu_\mathfrak{k}$. 
Argue as in the proof of Proposition~\ref{free action}, and write
$$p(X_j)=\sum_{h=1}^{n-k}a_{jh}p(X_h),\quad j=n-k+1,\ldots,d.$$
Recall that the $d-n+k$ vectors $Y_j=X_j-\sum_{h=1}^{n-k}a_{jh}X_h$, $j=n-k+1,\ldots,d$, define a set of generators of $\mathfrak{k}$.
Notice that, for each $j=n-k+1,\ldots,d$, (\ref{mm}) implies
$$\begin{array}{ccl}\langle \Phi_\mathfrak{k}([w_1:\cdots:w_{d}]),Y_j\rangle&=&\langle \Phi([w_1:\cdots:w_{d}]),j(Y_j)\rangle\\
&=&\langle \Phi([w_1:\cdots:w_{d}]),X_j-\sum_{h=1}^{n-k}a_{jh}X_h\rangle\\
&=&|w_j|^2+\lambda_j-\sum_{h=1}^{n-k}a_{jh}(|w_h|^2+\lambda_h).
\end{array}
$$
Therefore $[w_1:\cdots:w_{d}]\in\Phi_\mathfrak{k}^{-1}(0)$ if and only if
\begin{equation}\label{zerosetnew}
|w_j|^2=\sum_{h=1}^{n-k}a_{jh}(|w_h|^2+\lambda_h)-\lambda_j,\quad j=n-k+1,\ldots,d.
\end{equation}
Consider 
$R_j=e_j-\sum_{h=1}^{n-k}a_{jh}e_h\in\R^{d_\mathfrak{k}}$, with $j=n-k+1,\ldots,d$; notice that $\pi(R_j)=Y_j\in\mathfrak{k}$ and therefore
$R_j\in \ker(\pi_{\mathfrak k})$.
Choose 
$r_{d_\mathfrak{k}+1},\ldots,r_d$ so that the vector 
$R=\sum_{j=d_{\mathfrak{k}}+1}^{d}r_jR_j$, satisfying $\pi(R)\in\mathfrak{k}$, verifies
$$\exp(R)[w_1:\cdots:w_d]= 
[\tilde{w_1}:\cdots:\tilde{w}_d],$$
where $\tilde{w}_{d_{\mathfrak{k}+1}},\ldots,\tilde{w}_d$ are nonnegative real numbers. Notice that we have acted with an
element that projects to $K$, thus 
$[[w_1:\cdots:w_d]]=[[\tilde{w_1}:\cdots:\tilde{w}_d]]$ in $M_\k$.
Since $\tilde{w_1},\ldots,\tilde{w}_{d_\mathfrak{k}}$ satisfy (\ref{zerosetnew}), we have that
$\langle J_\mathfrak{k}(\tilde{w_1},\ldots,\tilde{w}_{d_\mathfrak{k}}),R_j\rangle=0$, for 
$j=n-k+1,\ldots,d_\mathfrak{k}$. The vectors $R_j$, for $j=n-k+1,\ldots,d_\k$, form
a basis of $\ker(\pi_k)$.
Hence 
$\Psi_{\mathfrak{k}}(\tilde{w_1},\ldots,\tilde{w}_{d_\mathfrak{k}})=(i^*_\mathfrak{k}\circ J_\mathfrak{k})(\tilde{w_1},\ldots,\tilde{w}_{d_\mathfrak{k}})=0$
and, therefore,
$[\tilde{w_1}:\cdots:\tilde{w}_{d_\mathfrak{k}}]\in X_{\mathfrak{k}}$.

Let us finally show that, if $[[\tilde{\vw}]]=[[\hat{\vw}]]\in M_{\mathfrak{k}}$
are such that $\tilde{w}_j,\hat{w}_j\in[0,\infty)$, for $j=d_{\mathfrak{k}}+1,\ldots,d$, then
$[\tilde{w}_1:\cdots:\tilde{w}_{d_\mathfrak{k}}]=
[\hat{w}_1:\cdots:\hat{w}_{d_\mathfrak{k}}]$ in $X_\k$. Since $[[\tilde{\vw}]]=[[\hat{\vw}]]$ there exist $X,Y\in\R^d$, with 
$\exp(X)\in N$ and $\Pi(\exp(Y))\in K$, such that
$$\exp(X)\exp(Y)(\tilde{w}_1,\ldots,\tilde{w}_{d_\mathfrak{k}},\tilde{w}_{d_\mathfrak{k}+1},\ldots,\tilde{w}_d)=(\hat{w}_1,\ldots,\hat{w}_{d_\mathfrak{k}}\hat{w}_{d_\mathfrak{k}+1},\ldots,\hat{w}_d).$$
Notice that $\pi(X)\in Q$ and $\pi(Y)\in \mathfrak{k}$. Moreover, since $\tilde{w}_j$ and $\hat{w}_j$, $j=d_\mathfrak{k}+1,\ldots,d$, are 
nonnegative real numbers by hypothesis, we have necessarily that $\tilde{w}_j=\hat{w}_j$ and that there exists $\underline{n}\in\Z^d$ such that
$X+Y+\underline{n}\in\R^{d_\mathfrak{k}}\times\Z^{d-d_\mathfrak{k}}$. Notice that
$\pi(X+\underline{n})\in Q$ and $\pi(Y)\in \mathfrak{k}$. Therefore $\exp(X+\underline{n}+Y)\in N_\mathfrak{k}$ 
and
$\exp(X+\underline{n}+Y)(w_1,\ldots,w_{d_\mathfrak{k}})=(\tilde{w}_1,\ldots,\tilde{w}_{d_\mathfrak{k}})$. Thus
$[\tilde{w}_1:\cdots:\tilde{w}_{d_\mathfrak{k}}]=[\hat{w}_1:\cdots:\hat{w}_{d_\mathfrak{k}}]$.
\qed\end{proof}

\noindent\textbf{Proof of Theorem~\ref{reduction}.}\quad
We define a collection of quasifold charts for the orbit space $M_\mathfrak{k}=\Phi_{\k}^{-1}(0)/K$ as follows. 
Let $\nu_{\k}$ be a vertex of $\Delta_\k$ and let $\nu$ be a vertex of 
$\Delta$ lying in the closure of the smallest face $F$ of $\Delta$ containing $\nu_\k$.
By Proposition~\ref{free action},  we can order the normal vectors $\xd$ for $\Delta$ so that:
\begin{itemize}
\item $\{p(X_1),\ldots,p(X_{d_\k})\}$ are normal vectors for $\Delta_\k$;
\item $\{p(X_1),\ldots,p(X_{n-k})\}$ are the normal vectors for $\Delta_\k$ corresponding to the facets  that meet at $\nu_\mathfrak{k}$;
\item $X_1,\ldots,X_{n-k},X_{d_\k+1},\ldots,X_{d_\k+k}$ are the normal vectors for $\D$ corresponding to the facets that meet at $\nu$. 
\end{itemize}
Remark that the face $F$ is given by the intersection of the hyperplanes
corresponding to $\{X_1,\ldots,X_{n-k}\}$ and has dimension $k$. Moreover, notice that
Proposition~\ref{free action} implies that $\nu\notin\Delta_\k$. 

Consider the non--empty open subset of $\C^{n-k}$ given by
$$\tilde{U}_{\nu_\k}=\left\{\vz\in\C^{n-k}\,|\,
\sum_{h=1}^{n-k}a_{jh}(|z_h|^2+\lambda_h)-\lambda_j>0,\;j=n-k+1,\ldots,d_\k\right\}.$$
We send it to the chart $(U_{\nu},\rho_{\nu},\tilde{U}_{\nu}/\Gamma_\nu)$
of $M$ corresponding to the vertex $\nu$ by the continuous, equivariant map:
$$\begin{array}{ccccc}
\tilde{\Psi}^{\nu}_\k&:&\tilde{U}_{\nu_\k}&\longrightarrow&\tilde{U}_{\nu}\\
&&(z_1,\ldots,z_{n-k})&\longmapsto&\left(z_1,\ldots,z_{n-k},w_{d_\k+1},\ldots,w_{d_\k+k}\right),
\end{array}
$$
where the $w_j$'s are given by
$$
w_j=\sqrt{\sum_{h=1}^{n-k}a_{jh}(|z_h|^2+\lambda_h)-\lambda_j},\quad j=d_\mathfrak{k}+1,\ldots,d.$$
It is straightforward to check that the induced mapping
$$
\begin{array}{ccccc}
\Psi^{\nu}_\k&:&\tilde{U}_{\nu_\k}/\Gamma_{\nu}^\k&\longrightarrow & \Big(\tilde{U}_{\nu}/\Gamma_{\nu}\Big)\cap\rho_{\nu}^{-1}(\Phi_\k^{-1}(0))\\
&&[z_1:\cdots:z_{n-k}]&\longmapsto&
[z_1:\cdots:z_{n-k}:w_{d_\k+1}:\cdots:w_{d_\k+k}]
\end{array}
$$
is a homeomorphism; the group $\Gamma_{\nu}$ here is given by 
$$N\cap[(S^1)^{n-k}\times\{1\}^{d_\k-n+k}\times(S^1)^k\times\{1\}^{d-d_{\k}-k}],$$ while $\Gamma_{\nu}^\k$
is the subgroup of $\Gamma_{\nu}$ acting on $\tilde{U}_{\nu_\k}$  that is given by
$$N\cap[(S^1)^{n-k}\times\{1\}^{d-n+k)}].$$
Consider now the countable group $\Gamma_{\nu_k}=N_\k\cap[(S^1)^{n-k}\times\{1\}^{d_\k-n+k}]$.
By the same argument used at the end of the proof of Lemma~\ref{rotazioni}, we have that
$\Gamma_\nu^\k$ is a subgroup of $\Gamma_{\nu_\k}$ and the mapping
$$
\begin{array}{ccccc}
\rho_{\nu}\circ\tilde{\Psi}^{\nu}_\k&:&\tilde{U}_{\nu_\k}/\Gamma_{\nu_\k}&\longrightarrow & \frac{U_{\nu}\cap \Phi_\k^{-1}(0)}{K}\\
&&[z_1:\cdots:z_{n-k}]&\longmapsto&[z_1:\cdots:z_{n-k}:w_{n-k+1}:\cdots:w_d]
\end{array}
$$
is itself a homeomorphism. The explicit expression of the mapping $\rho_{\nu}\circ\tilde{\Psi}^{\nu}_\k$ is determined
by the mapping $\rho_\nu$ given in (\ref{rhonu}).
Repeating this procedure for each vertex  $\nu_{\k}$ of $\Delta_\k$ yields an atlas of $M_\k=\Phi_\k^{-1}(0)/K$.

Now, using Lemma~\ref{rotazioni}, we can define the following mapping
$$
\begin{array}{ccccc}
g&\colon&M_\mathfrak{k}&\rightarrow&X_\mathfrak{k}\\
&&[[\vw]]&\longmapsto&[\tilde{w}_1:\cdots:\tilde{w}_{d_{\mathfrak{k}}}]
\end{array}
$$

\noindent\underline{$g$ is injective:} let $g([[\vw]])=g([[\vw']])$, then there exist
$\tilde{\vw}$ and $\tilde{\vw}'$ as in the statement of Lemma~\ref{rotazioni} such that 
$[\tilde{w}_1:\cdots:\tilde{w}_{d_{\mathfrak{k}}}]=
[\tilde{w}'_1:\cdots:\tilde{w}'_{d_{\mathfrak{k}}}]$. Then there exists
$W\in\R^{d_{\mathfrak{k}}}\times\{0\}^{d-d_{\mathfrak{k}}}$ such that $\pi(W)\in Q+\mathfrak{k}$ 
and $\exp(W)(\tilde{w}_1,\ldots,\tilde{w}_{d_{\mathfrak{k}}})
=(\tilde{w}'_1,\ldots,\tilde{w}'_{d_{\mathfrak{k}}})$. We can write $W=X+Y$ with $X,Y\in\R^d$ such that $\pi(X)\in Q$, $\pi(Y)\in \mathfrak{k}$. 
Moreover, by (\ref{zerosetnew}), $\tilde{w}_j=\tilde{w}'_j$, $j=d_{\mathfrak{k}}+1,\ldots,d$. 
Thus $\exp(X)\exp(Y)(\tilde{w}_1,\ldots,\tilde{w}_{d})=(\tilde{w}'_1,\ldots,\tilde{w}'_{d})$ with
$\exp(X)\in N$ and $\exp(Y)\in K$; therefore 
$[[\vw]]=[[\tilde{\vw}]]=[[\tilde{\vw}']]=[[\vw']]$.

\noindent\underline{$g$ is surjective:} let $[z_1:\cdots:z_{d_{\mathfrak{k}}}]\in X_{\mathfrak{k}}$. Consider first, as before, 
$R_j=e_j-\sum_{h=1}^{n-k}a_{jh}e_h\in\R^{d_\mathfrak{k}}$, with $j=n-k+1,\ldots,d$; notice that $\pi(R_j)=Y_j\in\mathfrak{k}$ and therefore
$R_j\in\mathfrak{n}_\mathfrak{k}$. By (\ref{delzantspace}) we have that 
$i^*_\mathfrak{k}\circ J_\mathfrak{k}([z_1:\cdots:z_{d_\mathfrak{k}}])=0$. This implies
$$\langle J_\mathfrak{k}([z_1:\ldots:z_{d_\mathfrak{k}}],R_j\rangle=0,\quad\forall j=d-n+k,\ldots,d_\mathfrak{k},$$
which, in turn, implies (\ref{zerosetnew}) for $j=d-n+1,\ldots,d_\mathfrak{k}$.
Take 
$$
w_j=\sqrt{\sum_{h=1}^{n-k}a_{jh}(|z_h|^2+\lambda_h)-\lambda_j},\quad j=d_\mathfrak{k}+1,\ldots,d.
$$
Then $[z_1:\cdots,z_{d_\mathfrak{k}}:w_{d_\mathfrak{k}+1}:\cdots:w_d]\in M_{\mathfrak{k}}$ and its image under the mapping $g$ is exactly
$[z_1:\cdots:z_{d_{\mathfrak{k}}}]$. Therefore $$g^{-1}([z_1:\cdots:z_{d_{\mathfrak{k}}}])=
[z_1:\cdots,z_{d_\mathfrak{k}}:w_{d_\mathfrak{k}+1}:\cdots:w_d].$$
Since the mapping $g$ restricted to each chart is the identity, we can conclude that it is an equivariant diffeomorphism. 
Consider the following diagram
$$\xymatrix{\Phi_\k^{-1}(0)\ar[d]_{\rho}\ar[r]^{\iota}&(M,\omega)\\
\left(M_{\k},\omega_{\k}\right)
}.$$
One can verify, by the explicit local identifications above, that the reduced symplectic form $\omega_{\k}$ satisfies, as in the smooth case, 
$\rho^*\omega_{\k}=\iota^*\omega$, and that $g$ is a symplectomorphism.
Moreover, the mapping induced by $j^*\circ\Phi$ on $M_{\k}$ is the moment mapping for the action of $D^n/K$. \qed
\section{Some applications}
If $\D\subset\rndu$ is a smooth pointed polyhedron, one can choose normal vectors, $\xd$, 
for $\D$ that are primitive in $\Z^n$.
Then Theorem~\ref{thmp1} applied to $(\D, \{\xd\}, \Z^n)$
yields a symplectic toric manifold (when $\D$ is a polytope, this is the classical compact Delzant space corresponding to $\D$).
Moreover, in this case, $K=\mathfrak{k}/\mathfrak{k}\cap \Z^n$ is a general Lie 
subgroup of the torus $T^n=\rn/\Z^n$; it is a torus itself if, and only if, $\hbox{span}_{\R}(\mathfrak{k}\cap \Z^n)=\mathfrak{k}$.
Thus Theorem~\ref{reduction} yields the following
\begin{cor}\label{liscio}
Let $\D\subset\rndu$ be a smooth pointed polyhedron and let $M$ be the corresponding symplectic toric manifold. 
Consider the induced action on $M$ of any Lie subgroup $K$ of $T^n$, with moment mapping $\Phi_\mathfrak{k}=j^*\circ\Phi$. 
Assume that $K$ acts on $\Phi_\mathfrak{k}^{-1}(0)$ with $0$--dimensional isotropy groups.
Then the orbit space $M_\mathfrak{k}=\Phi_\mathfrak{k}^{-1}(0)/K$ is a symplectic quasifold of dimension $2(n-k)$, 
acted on by the quasitorus $T^n/K$. 
Moreover, $M_\mathfrak{k}$ is equivariantly symplectomorphic to the
symplectic toric quasifold $X_\k$ corresponding to $(\D_\mathfrak{k},\{p(X_1),\ldots,p(X_{d_\mathfrak{k}})\}, p(\Z^n))$. 
\end{cor}
Notice that, whenever $K$ is a torus, $M_\mathfrak{k}$ is the usual symplectic reduced space: it as an orbifold in general, a manifold if
the $K$--isotropy groups are all trivial.
\begin{example}[Reducing $\C\times S^2$ with respect to any Lie subgroup of the $2$--torus]{\rm 
Consider the strip $\D=[-1,\infty)\times [0,1]\subset(\R^2)^*$. It is an elementary example of a smooth pointed polyhedron. 
If we apply Theorem~\ref{thmp1} to the triple 
\begin{equation}\label{tripla}\Big([-1,\infty)\times [0,1], \{(1,0),(0,1),(0,-1)\},\Z^2\Big),\end{equation}
we obtain the noncompact toric manifold $\C\times S^2$, endowed with the standard symplectic structure. The first factor of $S^1\times S^1$ 
acts on $\C$ linearly with weight $1$, while the second factor acts on $S^2$ by rotations around the $z$--axis. 
This action is Hamiltonian and the image of the corresponding moment mapping is $[-1,\infty)\times [0,1]$. Now consider, 
for any positive real number $a$, the line $\mathfrak{k}=\hbox{span}_{\R}\{(-1,a)\}\subset \R^2$ (see Figure~\ref{quasihirzefan}).
\begin{figure}[h]
\begin{center}
\includegraphics[width=5cm]{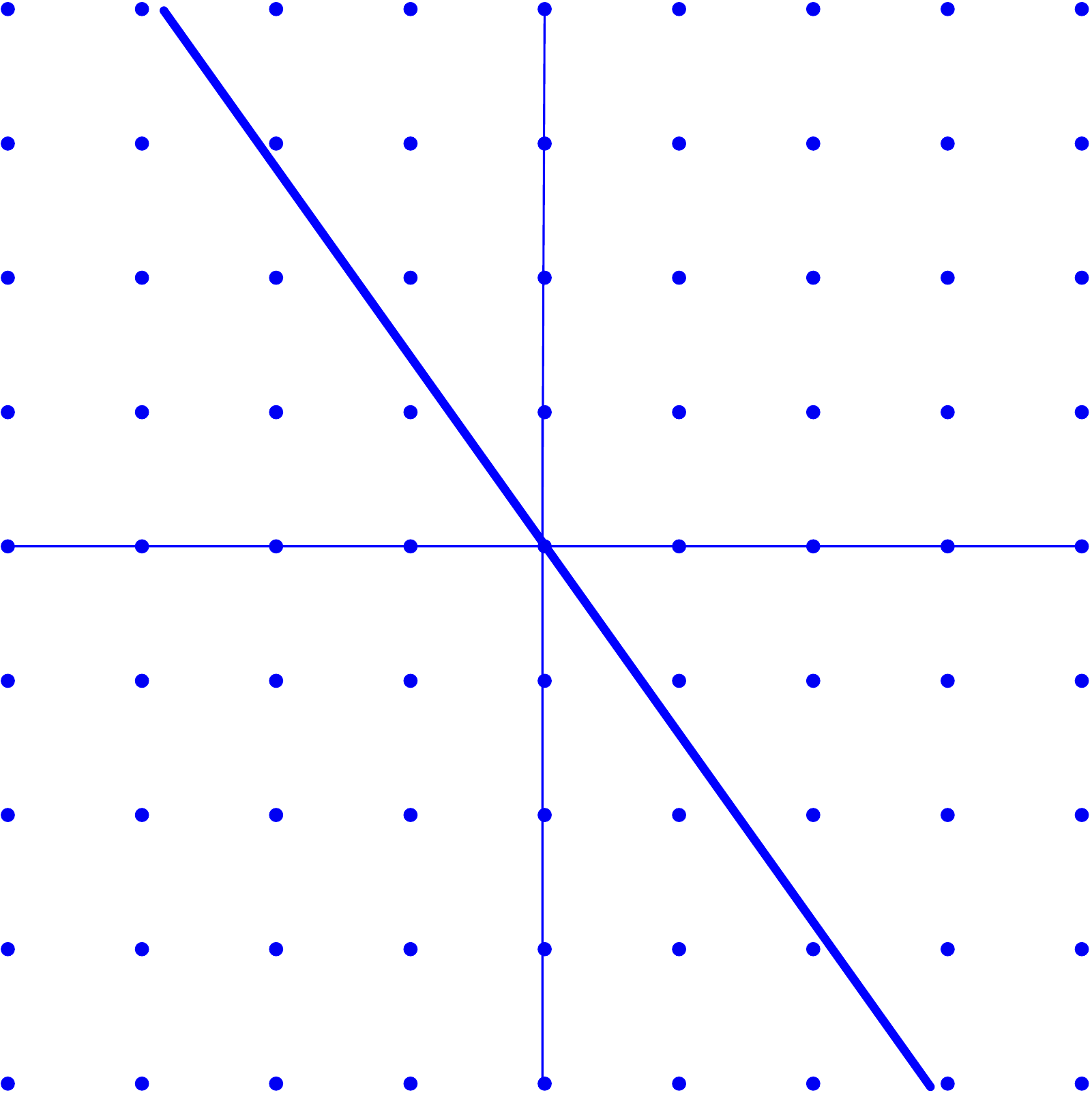}
\caption{The line $\mathfrak{k}$}
\label{quasihirzefan}
\end{center}
\end{figure}
The Lie subgroup $K=\k/(\k\cap \Z^2)$ is a circle if, and only if, $a$ is rational; otherwise it is the classical irrational 
wrap on $S^1\times S^1$. We want to reduce $\C\times S^2$ with respect to $K$, following Corollary~\ref{liscio}.
It is easy to check that the induced action of $K$ on $\Phi_\mathfrak{k}^{-1}(0)$
has $0$--dimensional isotropy groups. Notice that $\ker(j^*)$ can be identified with the line $x=ay$; therefore, $\D_\k$ is given by the segment
in Figure~\ref{hirzecut}.
\begin{figure}[h]
\begin{center}
\includegraphics[width=8cm]{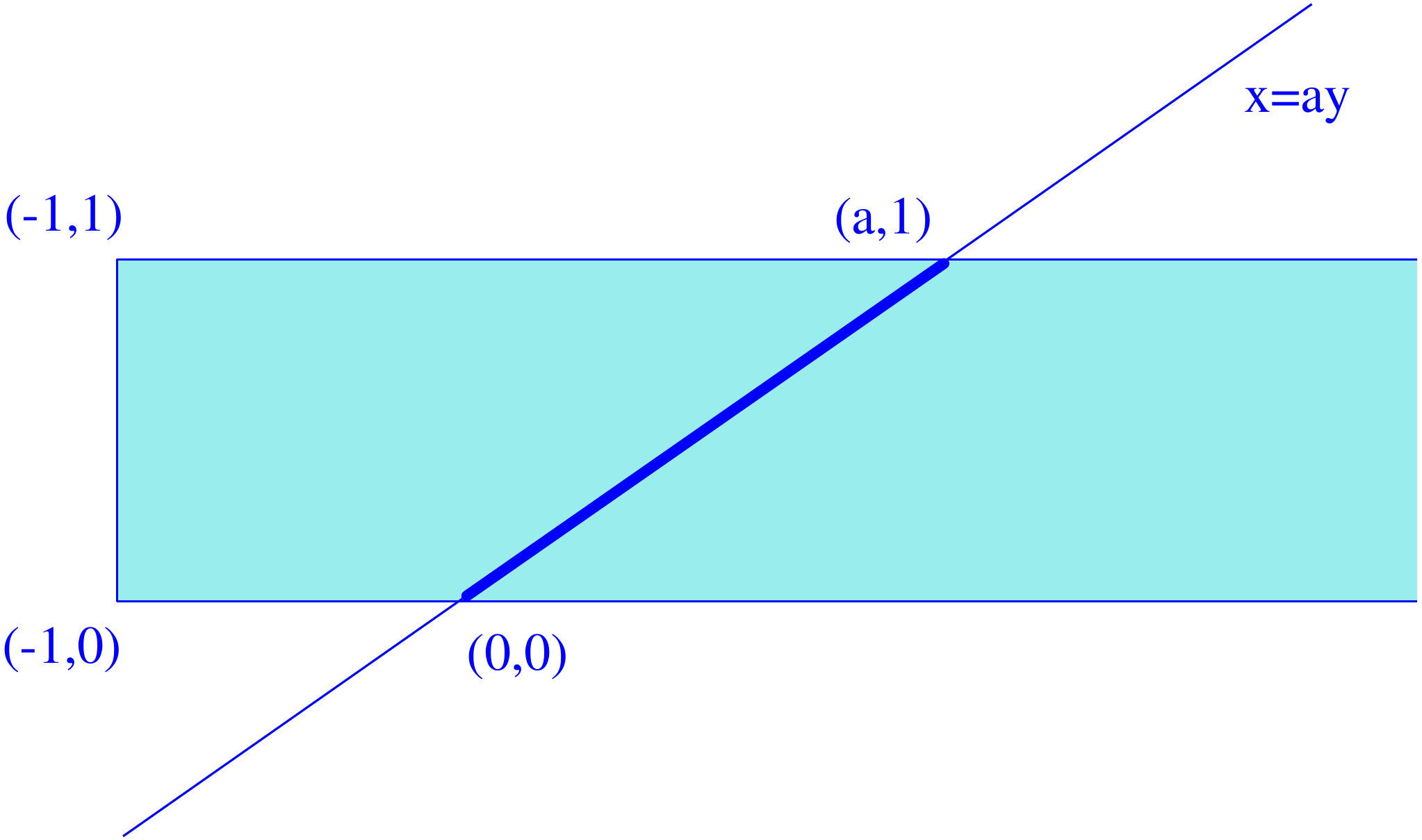}
\caption{The segment $\D_\k$}
\label{hirzecut}
\end{center}
\end{figure} 

Before we go on, it will be convenient to make the following identification. Let $f\,\colon\,\R^2/\k\longrightarrow\R$ be the linear isomorphism 
defined by $f(p(0,1))=1$. Then 
$f(p(0,-1))=-1$ and $f(p(1,0))=a$. Let $p(0,1)^*$
the basis of $(\R^2/\k)^*$ dual to $p(0,1)$. Then $p^*(p(0,1)^*)=(a,1)\in \ker j^*\subset(\R^2)^*$. Therefore, from (\ref{tripla}) we get that $\Delta_\k$ is sent to
$$\{\mu\in\R^*\;|\;\langle \mu,a\rangle\geq -1\}\cap
\{\mu\in\R^*\;|\;\langle \mu,1\rangle\geq 0\}\cap
\{\mu\in\R^*\;|\;\langle \mu,-1\rangle\geq -1\}.$$
We discard the first half--line, since its interior contains the intersection of the
remaining two. We obtain that $\D_\k\simeq [0,1]$ and that the corresponding triple is given by 
$$([0,1],\{1,-1\},\Z+a\Z).$$
If we apply Theorem~\ref{thmp1}, we find the quasifold 
$$X_\k=\frac{\{(z_1,z_2)\in\C^2\;|\;|z_1|^2+|z_2|^2=1\}}{\{(e^{2\pi i(t+am)},e^{2\pi i t})\in S^1\times S^1\;|\;(t,m)\in\R\times\Z\}}\simeq \frac{S^2}{\Gamma_a},$$
where $\Gamma_a=(\Z+a\Z)/\Z\simeq\{e^{2\pi i a m}\;|\;m\in\Z\}$. 
The quasifold $X_\k$ is acted on by the $1$--dimensional quasitorus $\R/(\Z+a\Z)$, while the reduced symplectic quasifold $M_\k$ is endowed with the residual action of the $1$--dimensional quasitorus $T^2/K\simeq\R/(\Z+a\Z)$.
By Corollary~\ref{liscio}, the quasifolds $X_\k$ and $M_\k$ are equivariantly symplectomorphic. For an irrational number $a$, the quasifold $X_\k$ can be viewed as a nonrational counterpart of $S^2$, 
similarly to the {\em quasisphere} introduced in \cite[Examples 1.13, 3.5]{p}.
In \cite[Example~2.4.3]{bz} the quotient $X_\k$ was obtained as the leaf space of a holomorphic foliation.  Moreover, $X_\k$ arises in the construction
of a one--parameter family of quasifolds $\F_a$ that contains all of the Hirzebruch surfaces \cite[Section~3]{hirze}. 
In fact, the spaces $\F_a$ turn out to be equal to the disjoint union of a dense open subset and of $X_\k$.
They are obtained by cutting the symplectic manifold $\C\times S^2$ in the direction $(-1,a)$; this amounts to cutting the
above strip with the line $x=ay$. Since standard cutting only works when the number $a$ is rational, we use a generalization of this procedure for
nonrational simple pointed polyhedra (see Remark~\ref{cutting}).
}
\end{example}

\begin{remark}[Diffeologies]{\rm The $1$--dimensional quasitorus $\R/(\Z+a\Z)$ was studied by Donato and Iglesias within the theory of diffeological spaces \cite{patrick1,patrick2}; in this setting Iglesias introduced the terminology {\em irrational torus}.}
\end{remark}

\begin{remark}[Nonrational symplectic cutting]\label{cutting}
{\rm In our article \cite{cut} we have extended symplectic cuts \cite{lerman} and blow--ups to symplectic toric quasifolds. In doing so, we have
already implicitly used symplectic quotients, and proved with a direct argument that cuts are symplectic quasifolds.
One could also proceed by applying Theorem~\ref{reduction} above. However, the direct approach in \cite{cut} 
is preferable, since it also allows cutting through vertices
of the original simple pointed polyhedron. In this case, the hypothesis of Theorem~\ref{reduction} that 
$K$ acts on $\Phi_\mathfrak{k}^{-1}(0)$ with $0$--dimensional isotropy groups is 
no longer satisfied; however, as it turns out, the quotient is still a symplectic quasifold, 
mirroring what happens in the smooth case \cite{gs1}.}
\end{remark}
\section*{Acknowledgements}
This research was partially supported by grant PRIN 2015A35N9B\_$\!$\_013 (MIUR, Italy) and by GNSAGA (INdAM, Italy).

\sc{Dipartimento di Matematica e Informatica "U. Dini"\\
Universit\`a degli Studi di Firenze \\
Viale Morgagni 67/A\\
50134 Firenze, ITALY\\}
{\it E-mail addresses:} \tt{fiammetta.battaglia@unifi.it, elisa.prato@unifi.it}

\end{document}